\documentclass[10pt]{amsart}
\usepackage{amsmath}
\usepackage{amstext}
\usepackage{amssymb}
\usepackage{amsthm}
\usepackage{layout}
\usepackage{epic}
\usepackage{amscd}
\usepackage[T2A,T1]{fontenc}
\usepackage[all,cmtip]{xy}
\DeclareSymbolFont{cyrillic}{T2A}{cmr}{m}{n}
\DeclareMathSymbol{\Sha}{\mathalpha}{cyrillic}{216}
\theoremstyle{plain}
\newtheorem{thm}{Theorem}[section]
\newtheorem{cor}[thm]{Corollary}
\newtheorem{lem}[thm]{Lemma}
\newtheorem{prop}[thm]{Proposition}

\newtheorem*{thm*}{Theorem 1.1}
\theoremstyle{definition}
\newtheorem{defn}[thm]{Definition}
\newtheorem{expl}[thm]{Example}
\newtheorem{rem}[thm]{Remark}
\newtheorem*{note}{Notation}

\DeclareMathOperator{\GL}{GL} \DeclareMathOperator{\Ker}{ker}

\DeclareMathOperator{\SL}{SL} \DeclareMathOperator{\PGL}{PGL}
 
\DeclareMathOperator{\Hom}{Hom} \DeclareMathOperator{\image}{Im}
\DeclareMathOperator{\Sym}{Sym} 
 \DeclareMathOperator{\Gal}{Gal}

\DeclareMathOperator{\Frob}{Frob}
\DeclareMathOperator{\ad}{ad}\DeclareMathOperator{\End}{End}
\DeclareMathOperator{\Ind}{Ind}
\newcommand{\oo}{\mathfrak{o}}
\newcommand{\OO}{\mathcal{O}}
\newcommand{\QQ}{\mathbf{Q}}
\newcommand{\RR}{\mathbf{R}}

\newcommand{\ZZ}{\mathbf{Z}}
\newcommand{\CC}{\mathbf{C}}
\newcommand{\F}{\mathbf{F}}

\newcommand{\DD}{\mathbf{D}}

\newcommand{\p}{\mathfrak{p}}

\newcommand{\ab}{\mathrm{ab}}

\newcommand{\crys}{\mathrm{crys}}
\newcommand{\s}{\mathrm{ss}}

\newcommand{\cln}{\colon}
\newcommand{\mf}[1]{\mathfrak{{#1}}}

\newcommand{\ol}[1]{\bar{{#1}}}
\numberwithin{equation}{thm} 
\begin{document}
\title{\bf{Unobstructed Hilbert modular deformation problems}}
\author{Adam Gamzon}
\date{}

\begin{abstract}
Let $\rho_{f,\lambda}$ be the Galois representation associated to a Hilbert newform $f$.  Consider its semisimple mod $\ell$ reduction $\bar{\rho}_{f,\lambda}$.  This paper discusses how, under certain conditions on $f$, the universal ring for deformations of $\bar{\rho}_{f,\lambda}$ with fixed determinant is unobstructed for almost all primes.  We follow the approach of Weston, who carried out a similar program for classical modular forms in 2004.  As such, the problem essentially comes down to verifying that various local invariants vanish at all places dividing $\ell$ or the level of the newform.  We conclude with an explicit example illustrating how one can in principle find a lower bound on $\ell$ such that the universal ring for deformations of $\bar{\rho}_{f,\lambda}$ with fixed determinant is unobstructed for all $\lambda$ over $\ell$.
\end{abstract}

\maketitle

\section{Introduction}
%\footnote{The author's research was supported by a Fulbright Post-doctoral Fellowship, awarded by the Fulbright Commission in Israel, the United States-Israel Educational Foundation and by ERC grant.}
%flesh out intro and put in correct

Let $f$ be a newform of level $N$ and weight $k \geq 2$.  Let $K_f$ be the number field obtained from $f$ by adjoining its Hecke eigenvalues to $\QQ$.  For each prime $\lambda$ in $K_f$, Deligne constructed a semisimple mod $\ell$ representation $\bar\rho_{f,\lambda}$.  In \cite{Mazur2}, Mazur conjectured that the universal deformation ring of this residual representation $\ol{\rho}_{f,\lambda}$ is unobstructed for almost all $\lambda$.  Weston \cite{Weston2} gave a positive answer to Mazur's question in 2004 assuming that $k \geq 3$.  He was also able to obtain some results for weight two modular forms, showing that Mazur's conjecture holds on a set of primes of density one.   We show that Weston's methodology and results essentially carry over to the Hilbert modular form setting with a few minor adjustments.

More specifically, let $F$ be a totally real extension of $\QQ$ of degree $d>1$ and let $f$ be a Hilbert newform on $F$ of  level $\mathfrak{n} \subset \OO_{F}$ and weight $k = (k_{\tau_1},\dots, k_{\tau_d})$.  Here the $\tau_i$ denote the embeddings of $F$ into $\RR$.  We assume that $k_{\tau_i}\geq 2$ for all $i$ and that they satisfy the parity condition $k_{\tau_1} \equiv \cdots \equiv k_{\tau_d} \bmod 2$.  As in the previous paragraph, let $K_f$ be the number field generated over $\QQ$ by the Hecke eigenvalues of $f$ and let $\OO_{K_f}$ its ring of integers.  For each prime $\lambda$ of $K_f$, let $$\ol{\rho}_{f,\lambda}\cln G_{F,S} \rightarrow \GL_2(k_{f,\lambda})$$ be the semisimple mod $\ell$ Galois representation attached to $f$ by Carayol and Taylor.  Here $k_{f,\lambda} = \OO_{K_f}/\lambda$ and $G_{F,S} = \Gal(F_S,F)$, where $F_S$ is the maximal algebraic extension of $F$, unramified outside of a finite set of places $S = \{v |\mathfrak{n}\ell\} \cup \{v |\infty\}$.

Let $\DD^{\det=\delta}_{\ol{\rho}_{f,\lambda}}$ denote the functor that associates to a coefficient ring $R$ the set of all deformations of $\ol{\rho}_{f,\lambda}$ to $R$  with fixed determinant (see section 2 for precise definitions regarding deformation theory).  Note that $\ol{\rho}_{f,\lambda}$ is absolutely irreducible for almost all $\lambda$ \cite[Proposition 3.1]{Dimitrov1}.  For such $\lambda$, the functor $\DD^{\det = \delta}_{\ol{\rho}_{f,\lambda}}$ is representable by the \emph{universal deformation ring} $R_{f,\lambda}$ for deformations with fixed determinant.  Then our main theorem is the following.  

\begin{thm}\label{mainthm}
Set $k_0 = \max_i\{k_{\tau_i}\}$.  Suppose that $f$ has no CM, is not a twist of a base change of a Hilbert newform on $E \subsetneq F$, and $k_0 \geq 3$.  Then $R_{f,\lambda}$ is unobstructed for almost all $\lambda$.
\end{thm}

\begin{rem}
Weston \cite{Weston2} did not have this additional condition of deformations with fixed determinant, but in general there are obstructions that come from lifting the determinant, so there is no way around this.  See, for example, \cite[Theorem 10.7.3]{Neukirch} for details about calculating $\dim_{\F_p} H^2(G_{F,S}, \ZZ/p\ZZ)$.
\end{rem}

\begin{rem}
The hypotheses that $f$ has no CM and is not a twist of a base change come from ensuring that certain Selmer groups vanish for almost all $\lambda$ (see Proposition \ref{unobscriterion}, \cite[Theorem B(i)]{Dimitrov3} and \cite[Theorem 2.1]{Dimitrov4}).  It is an open problem as to whether or not these hypotheses can be relaxed.
\end{rem}

The strategy for proving Theorem \ref{mainthm} is to use a generalization of a criterion for unobstructedness (Proposition \ref{unobscriterion}) due to Weston \cite{Weston2}.  Using this proposition and results of Dimitrov (\cite{Dimitrov3} and \cite{Dimitrov4}), the proof is reduced to checking that for all $v\in S$, the local cohomology groups $H^0(G_v, \ol{\varepsilon}\otimes \ad^0 \ol{\rho}_{f,\lambda}) = 0$ for almost all $\lambda$.  Here $\varepsilon$ is the $\ell$-adic cyclotomic character, $G_v$ is a decomposition group a $v$ and $\ad^0 \rho$ denotes the restriction of the adjoint representation of $\rho$ to the trace-zero matrices.  Section 3 addresses those $v \in S$ such that $v \nmid \ell$, while section 4 shows that for almost all $\lambda$, this vanishing cohomology condition holds for $v |\ell$.  We also give a proof of Theorem \ref{mainthm} in section 4.  We conclude in Section 5 with an explicit example of determining a lower bound on $\ell$ such that $R_{f,\lambda}$ is unobstructed for all $\lambda$ over $\ell$.  Here $f$ is the unique newform on $\QQ(\sqrt{5})$ of weight (2,4) and level $(7/2 + \sqrt{5}/2)$.

It is with great pleasure that the author thanks Tom Weston for suggesting this problem and for several helpful suggestions along the way.  Many thanks are also owed to Mladen Dimitrov for patiently answering every question put to him, especially regarding the vanishing of the previously mentioned Selmer groups.  The author also benefited from a number of informative conversations with Ehud de Shalit and for this he is most grateful.  Finally, the author acknowledges with gratitude that this work was produced while he was jointly supported as a Fulbright postdoctoral fellow at the Hebrew University of Jerusalem and under the framework of the ERC grant entitled \emph{Langlands correspondence and its variants} under David Kazhdan.

\begin{note}
 For a field $F$, denote its absolute Galois group by $G_F$.  As above, we let $G_v$ denote a decomposition group at a place $v$ of $F$ and fix embeddings $G_v \hookrightarrow G_F$.  Let $F_v$ denote the $v$-adic completion of $F$.  We use the phrase ``almost all'' as a substitute for ``all but finitely many.''
\end{note}

\section{Review of Galois deformation theory}

We briefly recall the theory of deformations of mod $\ell$ Galois representations in the sense of Mazur.  For a more thorough introduction see \cite{Bockle} or \cite{Gouvea}.  

Let $F$ be a number field and let $S$ be a finite set of places of $F$.  Let $k$ be a finite field of characteristic $\ell$ and denote the Witt vectors of $k$ by $W(k)$.  Consider an absolutely irreducible continuous representation $$\bar{\rho}: G_{F,S} \rightarrow \GL_n(k).$$  Also consider the category $\mathcal{C}$ of complete local noetherian rings $R$ with residue field $k$.  Morphisms in this category are local homomorphisms that induce the identity on $k$.  A \emph{lift} of $\bar{\rho}$ to $R$ is a continuous representation $\rho: G_{F,S} \rightarrow \GL_n(R)$ making the following diagram commute:\\
\hspace*{\fill}\xymatrix{
  G_{F,S} \ar[r]^\rho \ar[rd]_{\bar{\rho}} & \GL_n(R) \ar[d]\\
  				 & \GL_n(k)
   }\hspace{\fill}
\\ where the homomorphism $\GL_n(R) \rightarrow \GL_n(k)$ is the map induced by the reduction homomorphism $R \rightarrow k$.  We say that two lifts $\rho$ and $\rho'$ of $\bar{\rho}$ to $R$ are \emph{strictly equivalent} if $\gamma\rho\gamma^{-1} = \rho'$ for some $\gamma \in \Ker(\GL_n(R) \rightarrow \GL_n(k))$.

\begin{defn}
A \emph{deformation} of $\bar{\rho}$ to $R$ is a strict equivalence class of lifts of $\bar{\rho}$ to $R$.
\end{defn}

Consider the functor $D_{\bar{\rho}}: \mathcal{C} \rightarrow SETS$ given by $$D_{\bar{\rho}}(R) = \{\mbox{deformations of $\bar{\rho}$ to $R$}\}.$$  Call such a functor a \emph{deformation problem}.

\begin{thm}[Mazur]\label{funddefthm}
If $\bar{\rho}$ is absolutely irreducible then $D_{\bar{\rho}} = \Hom(R_{\bar{\rho}}, -)$ and $$R_{\bar{\rho}} \cong W(k)[[x_1,\dots, x_{d_1}]]/I.$$  Here $d_i = \dim_kH^i(G_{F,S}, \ad \bar{\rho})$ and $I$ is generated by at most $d_2$ elements.
\end{thm}

\begin{defn}
The deformation problem $D_{\bar{\rho}}$ is \emph{unobstructed} if $d_2 = 0$.
\end{defn}

We can also consider subfunctors of $D_{\bar{\rho}}$ where we ask our deformations to satisfy certain prescribed properties.  For example, we can ask for deformations with fixed determinant.  By this we mean that $\det \rho$ is the composition of the canonical homomorphism $W(k) \rightarrow R$ (making $R$ a $W(k)$-algebra) with a fixed continuous character $\delta:G_{F,S} \rightarrow W(k)$.  When this occurs, we say that a deformation $\rho$ has $\det = \delta$.  Denote by $D_{\bar{\rho}}^{\det =\delta}$ the subfunctor given by $$D_{\bar{\rho}}^{\det=\delta} = \{\mbox{deformations of $\bar{\rho}$ to $R$ with $\det = \delta$}\}.$$  Note that for the deformation problem $D_{\bar{\rho}}^{\det=\delta}$, an analogue to Theorem \ref{funddefthm} holds where we replace $\ad\bar{\rho}$ by $\ad^0\bar{\rho}$ in the statement of the theorem.

We now specialize to two-dimensional residual representations $\bar{\rho}:G_{F,S} \rightarrow \GL_2(k)$.  Let $K$ be a finite extension of $\QQ_\ell$ and let $\OO$ be its ring of integers.  Assume that we have a (fixed) continuous representation $$\rho: G_{F,S} \rightarrow \GL_2(\OO)$$ lifting $\bar{\rho}$.  Set $V_\rho = K^3$ and $A_\rho = (K/\OO)^3$.  Give $V_\rho$ and $A_\rho$ a $G_F$-action via $\ad^0\rho$.  Let $V_\rho(1)$ denote the Tate-twist of $V_\rho$.  Finally, define the Selmer groups $H_f^1(G_F,V_\rho(1))$ and $H_f^1(G_V, A_\rho)$ in the sense of Bloch-Kato \cite{Bloch-Kato}.  Then we have the following criterion for unobstructedness.

\begin{prop}\label{unobscriterion}
Suppose 
\begin{itemize}
\item[$(1)$] $H^0(G_v, \ol{\varepsilon}\otimes\ad^0\ol{\rho}) = 0$ for all $v \in S$,

\item[$(2)$] $H_f^1(G_F, V_\rho(1)) = 0$,

\item[$(3)$] $H_f^1(G_F,A_\rho) = 0$.
\end{itemize}
Then $H^2(G_{F,S}, \ad^0\ol{\rho})=0$.  That is, $\DD_{\ol{\rho}}^{\det = \delta}$ is unobstructed.
\end{prop}

\begin{proof}
The argument follows mutatis mutandis as in the proof of Proposition 2.2 in \cite{Weston2}.
\end{proof}

Thus the strategy for proving Theorem \ref{mainthm} is clear.  For $\rho = \rho_{f,\lambda}$, we need to check that the hypotheses of Proposition \ref{unobscriterion} hold for almost all primes $\lambda$ of $K_f$.

\section{Local invariants for $\ell \neq p$}

Let $v$ be a prime over a rational $p \in \ZZ$.  In this section, we show that the local invariants $H^0(G_v, \ol{\varepsilon}\otimes\ad^0\ol{\rho})$ are zero for almost all $\lambda$ not dividing $p$.  We  separate the proof into two cases based on the local Langlands correspondence for $\GL_2(F_v)$.

Let $K$ be any number field with ring of integers $\OO$.  For all primes $\lambda$ of $\OO$ not dividing $p = v \cap \ZZ$, fix an isomorphism $\iota_\lambda\cln \CC \rightarrow \ol{K}_\lambda$ extending the inclusion $\OO \hookrightarrow \OO_\lambda$.

Let $L$ be a finite extension of $\QQ_p$.  We say that a continuous character $\chi\cln L \rightarrow \CC$ is of \emph{Galois-type} with respect to $\iota_\lambda$ if the character $\iota_\lambda\circ\chi$ extends to a continuous character $\chi_\lambda\cln G_L \rightarrow \ol{K}_\lambda$ via the dense embedding $L^\times \hookrightarrow G_L^\ab$ of local class field theory.  Call $\chi$ \emph{arithmetic} if $\chi(F^\times) \subset \ol{\QQ}^\times$.

Let $\pi$ be an irreducible admissible complex representation of $\GL_2(F_v)$.  Call $\pi$ arithmetic if it satisfies one of the following conditions:
\begin{itemize}
\item $\pi$ is a subquotient of an induced representation $\pi(\chi_1,\chi_2)$ where the $\chi_i\cln F_v^\times \rightarrow \CC^\times$ are arithmetic characters (i.e., $\pi$ is principal series or special, coming from arithmetic characters),

\item $\pi$ is the base change of an arithmetic quadratic character $\chi\cln L^\times \rightarrow \CC^\times$ where $L/F_v$ is a quadratic extension (i.e., $\pi$ is supercuspidal and comes from the base change of an arithmetic character),

\item $\pi$ is extraordinary.
\end{itemize}

\begin{lem}\label{princeseries}%principal series, supercuspidal
Let $\pi$ be an arithmetic irreducible admissible complex representation of $\GL_2(F_v)$.  Let $\{\rho_\lambda\cln G_v \rightarrow \GL_2(\ol{K}_\lambda)\}$ be a family of continuous representations for $\lambda$ not dividing $p$ such that $\pi$ and $\rho_\lambda$ are in Langlands correspondence with respect to $\iota_\lambda$ for all $\lambda$.  If $\pi$ is principal series or supercuspidal then $$H^0(G_v, \ol{\varepsilon}\otimes\ad^0\ol{\rho}_\lambda^\s) = 0$$ for almost all $\lambda$.
\end{lem}

\begin{proof}
This follows precisely as in \cite[Proposition 3.2]{Weston2}, so we do not repeat the argument here.
\end{proof}

\begin{cor}\label{princeseriescor}
We have $H^0(G_v, \ol{\varepsilon}\otimes\ad^0\ol{\rho}_\lambda) = 0$ for almost all $\lambda$.
\end{cor}

\begin{proof}
This is clear from Proposition \ref{princeseries} since $$\dim_{\ol{\F}_\ell}H^0(G_v, \ol{\varepsilon}\otimes\ad^0\ol{\rho}_\lambda) \leq \dim_{\ol{\F}_\ell}H^0(G_v, \ol{\varepsilon}\otimes\ad^0\ol{\rho}_\lambda^\s).$$
\end{proof}

Note that for $\rho_\lambda$ as in Lemma \ref{princeseries}, $$\dim_{\ol{F}_\ell}H^0(G_\p, \ol{\varepsilon}\otimes \ad^0\ol{\rho}_\lambda^{\s}) = 1$$ for almost all $\lambda$ when $\pi$ is either one-dimensional or special.  Although the stronger vanishing result fails when $\pi$ is either one-dimensional or special, we can show the sufficient (and desired) vanishing of $H^0(G_v,\bar{\varepsilon}\otimes\ad^0\bar{\rho}_\lambda)$ for almost all $\lambda$ by using a level-lowering argument.

Let $\rho_{f,\lambda}\cln G_{F,S} \rightarrow \GL_2(K_{f,\lambda})$ be the Galois representation attached to a Hilbert newform $f$ of level $\mathfrak{n}$, weight $k$ and character $\psi$ by Carayol \cite{Carayol} and Taylor \cite{Taylor}.  Write $\psi = \psi_f|\cdot|^{k_0-2}$ where $\psi_f$ is a character of finite order and $|\cdot|$ is the norm character.  Note that $\det\rho_{f,\lambda} = \psi_f^{-1}\varepsilon^{1-k_0}$ where here we use the fact that the norm character corresponds to the compatible system of $G_F$-characters $\{\varepsilon_\lambda:=\varepsilon\}_\lambda$ and $\psi_f$ also denotes by abuse of notation the corresponding Galois character.  (We find it more convenient to work with this cohomological normalization rather than the usual normalization.)

Let $\pi$ be the automorphic representation corresponding to $f$.  Write $\pi = \otimes' \pi_v$ for the decomposition of $\pi$ into its irreducible admissible complex representations of $\pi_v$ into $\GL_2F_v$.  Fixing isomorphisms $\iota_\lambda\cln \CC \rightarrow \bar{K}_{f,\lambda}$, Carayol \cite[Th\'{e}or\`{e}me B]{Carayol} showed that each $\pi_v$ is arithmetic and is in Langlands correspondence with $\rho_{f,\lambda}|_{G_v}$ for $\lambda$ not dividing $p$.

\begin{rem}
The irreducible admissible representation $\pi_v$ must be infinite dimensional so nothing is lost by assuming that $\pi_v$ is special (as opposed to one-dimensional) in what follows.
\end{rem}

Suppose that $\pi_v$ is special.  That is, suppose that it is the infinite dimensional quotient of $\pi(\chi|\cdot|,\chi)$ for some arithmetic character $\chi\cln F_v^\times \rightarrow \CC^\times$.  Then the corresponding Galois representation has the form:  $$\rho_{f,\lambda}|_{G_v} \cong \begin{pmatrix} \varepsilon\chi_\lambda & \ast\\ 0 & \chi_\lambda \end{pmatrix}$$ where $\ast$ is nonzero.

\begin{lem}\label{reductionform}
Suppose that $q^2 \not\equiv 1 \bmod \lambda$ where $q = \#(\OO_F/v)$.  Then $$\ol{\rho}_{f,\lambda}|_{G_v} \otimes \bar{k}_{f,\lambda} \cong \begin{pmatrix} \bar{\varepsilon} \bar{\chi}_\lambda & \nu \\ 0 & \ol{\chi}_\lambda \end{pmatrix}$$ and $\nu$ is ramified.
\end{lem}

\begin{proof}
It is clear that the semi-simplification of $\ol{\rho}_{f,\lambda}|_{G_v} \otimes \bar{k}_{f,\lambda}$ has the form $\ol{\varepsilon}\ol{\chi}_\lambda \oplus \ol{\chi}_\lambda$, so it suffices to show that $\ol{\rho}_{f,\lambda}|_{G_v} \otimes \bar{k}_{f,\lambda}$ is not of the form $$\begin{pmatrix} \ol{\chi}_\lambda & \nu \\ 0 & \ol{\varepsilon}\ol{\chi}_\lambda \end{pmatrix}$$ where $\nu$ is nontrivial.  It is straightforward to check that $\ol{\varepsilon}^{-1}\ol{\chi}_\lambda^{-1}\nu$ is a 1-cocycle in $H^1(G_v, \ol{k}_{f,\lambda}(-1))$.  Consider the inflation-restriction exact sequence $$H^1(G_{\F_q}, \ol{k}_{f,\lambda}(-1)) \rightarrow H^1(G_v, \ol{k}_{f,\lambda}(-1)) \rightarrow H^1(I_v, \ol{k}_{f,\lambda}(-1))^{G_{\F_v}}$$ where $I_v\subset G_v$ is the inertia subgroup.  An easy calculation shows that in general $H^1(G_{\F_q}, \ol{k}_{f,\lambda}(-1)) = 0$ and that $H^1(I_v, \ol{k}_{f,\lambda}(-1))^{G_{\F_v}} = 0$ when $q^2 \not\equiv 1\bmod \lambda$.
\end{proof}

\begin{lem}\label{semisimplevanishing}
Suppose $2(q^2-1)q \not\equiv 0 \bmod \lambda$.  Then $H^0(G_v,\ad^0\ol{\rho}_{f,\lambda}) \neq 0$ if and only if $\ad^0\ol{\rho}_{f,\lambda}|_{G_v}\otimes \ol{k}_{f,\lambda}$ is semi-simple.
\end{lem}

\begin{proof}
This is a straightforward matrix calculation using Lemma \ref{reductionform}.  For example, choose the basis $\left(\begin{smallmatrix} 1 & 0\\0 & -1\end{smallmatrix}\right), \left(\begin{smallmatrix}0 & 1 \\ 0&0 \end{smallmatrix}\right), \left(\begin{smallmatrix} 0&0 \\1 &0\end{smallmatrix}\right)$ of $\End(V)$ where $V$ is the 3-dimensional $k_{f,\lambda}$-vector space endowed with a $G_{F,S}$ action by $\ad^0\ol{\rho}_{f,\lambda}$.  Then $$\ad^0\ol{\rho}_{f,\lambda}|_{G_v}\otimes \ol{k}_{f,\lambda} \cong \begin{pmatrix} 1 & -2\ol{\chi}_\lambda^{-1}\nu & 0 \\ 0 & \ol{\varepsilon} & 0\\ \ol{\varepsilon}^{-1}\ol{\chi}_\lambda^{-1}\nu & -\ol{\varepsilon}^{-1}\ol{\chi}_\lambda^{-2}\nu^2 & \ol{\varepsilon}^{-1}\end{pmatrix},$$ so it is clear that if $\nu = 0$ then $H^0(G_v, \ad^0\bar\rho_{f,\lambda}) \neq 0$.  Conversely, if $\nu$ is nonzero then using the fact that it is ramified (Lemma \ref{reductionform}) while $\bar{\varepsilon}$ and $\bar{\chi}_\lambda$ are not, one checks that there are no Galois invariants.
\end{proof}

\begin{prop}\label{specialcase}
If $\pi_v$ is special then $$H^0(G_v, \ol{\varepsilon}\otimes\ad^0\ol{\rho}_{f,\lambda}) = 0$$ for almost all $\lambda$.
\end{prop}

\begin{proof}
Note that $\pi_v$ has central character $\chi^2|\cdot|$ where $\chi$ is an arithmetic character giving rise to $\pi_v$.  By the local Langlands correspondence, this yields the equality $\chi^2 = \psi_{f,v}^{-1}|\cdot|^{-k_0}$ where $\psi_{f,v}$ is the $v$-component of $\psi_f$.  Set $\chi_v' = \chi^{-1}|\cdot|^{-k_0/2}$.  Note that $\chi_v'$ has finite order.  Extend $\chi_v'$ to a Hecke character $\chi'$ and twist $f$ by $\chi'$ to get an eigenform $f\otimes\chi'$.  Let $f'$ denote the newform in the eigenspace spanned by $f\otimes \chi'$ and let $\pi'$ denote the corresponding automorphic representation.  Then the $v$-component of $\pi'$ is a subquotient of $\pi(\chi\chi_v'|\cdot|,\chi\chi_v')$.   In particular, $\chi\chi_v'$ is unramified at $v$, so $v$ divides the level $\mathfrak{n}'$ of $f'$ exactly once.

Suppose $\lambda$ does not divide $2q(q^2-1)$ and suppose that \begin{equation}\label{specialnonvanish} H^0(G_v, \ol{\varepsilon}\otimes\ad^0\ol{\rho}_{f,\lambda}) \neq 0.\end{equation}  Then Lemma \ref{semisimplevanishing} implies that $\ol{\rho}_{f,\lambda}|_{G_v} \otimes \ol{k}_{f,\lambda} \cong \ol{\varepsilon}\ol{\chi}_\lambda \oplus \ol{\chi}_\lambda$.  This means that $$\ol{\rho}_{f',\lambda}|_{G_v}\otimes \ol{k}_{f,\lambda} \cong (\ol{\rho}_{f,\lambda} \otimes \ol{\chi'}_\lambda)|_{G_v} \otimes \ol{k}_{f,\lambda} \cong \ol{\varepsilon}^{1-k_0/2}\oplus \bar\varepsilon^{-k_0/2},$$ so $\ol{\rho}_{f',\lambda}|_{G_v}\otimes \ol{k}_{f,\lambda}$ is unramified at $v$.  Since $\lambda$ does not divide $q^2 -1$, we have that $N_{F/\QQ}(v) \not\equiv 1\bmod \ell$, so we may apply \cite[Theorem 0.1]{Jarvis} to get a congruent eigenform $f''$ of level $\mathfrak{n}'/v$.  That is, we get a set of Hecke eigenvalues $\{a(\mf{m}, f'')\}$ such that $a(\mf{q}, f'') \equiv a(\mf{q}, f') \bmod \lambda$ for all $\mf{q}$ not dividing $\mf{n}'\ell$.  By strong multiplicity one, there are only finitely many sets of eigenvalues, each one corresponding to a newform of level dividing $\mathfrak{n}'/v$.  Therefore, if (\ref{specialnonvanish}) holds for infinitely many $\lambda$ then for some newform $g$ of level dividing $\mf{n}'/v$ and for all $\mf{q}$ not dividing $\mf{n}'$, $$a(\mf{q}, g) \equiv a(\mf{q}, f') \mod \lambda$$ for infinitely many $\lambda$.  We conclude that $a(\mf{q},g) = a(\mf{q},f')$ for all $\mf{q}$ not dividing $\mf{n}'$, so applying strong multiplicity one again shows that $g = f$, a contradiction.
\end{proof}

\section{Local invariants for $\ell = p$}

We now recall the theory of Fontaine-Laffaille.  Let $K$ be a finite unramified extension of $\QQ_\ell$ and let $E/\QQ_\ell$ be another finite extension containing $K$.  Let $\sigma$ be the frobenius automorphism on $K$.  Given an $E$-linear representation $V$ of $G_K$, define the finite free $E\otimes_{\QQ_\ell} K$-module $$D_\crys(V) = (B_\crys \otimes_{\QQ_\ell} V)^{G_K}$$  where $B_\crys$ is Fontaine's crystalline period ring.  Note that $D_\crys(V)$ comes with a decreasing filtration $\{D_\crys(V)^i\}_i$ such that $$\cap_i D_\crys(V)^i = 0 \mbox{ and } \cup_i D_\crys(V) = D_\crys(V).$$  In addition, $D_\crys(V)$ comes with a $1_E\otimes \sigma$-semilinear map $\varphi:D_\crys(V) \rightarrow D_\crys(V)$.  Call $V$ is \emph{crystalline} if $\dim_E V$ equals the rank of $D_\crys(V)$ as a $E\otimes_{\QQ_\ell} K$-module.  

Suppose $V$ is an $E$-linear crystalline $G_K$-representation with Hodge-Tate filtration in the interval $[-(a+\ell -1), -a]$.  Consider the category $\mathrm{MF}^{a,a+\ell}(\OO_E)$ of strongly divisible lattices in $D_\crys(V)$ whose objects consist of finite free $\OO:=\OO_E \otimes_{\ZZ_\ell} \OO_K$-lattices $L \subset D_\crys(V)$ with a filtration $\{L^i:=L\cap D_\crys(V)^i\}$ and $1_{\OO_E}\otimes \sigma$-semilinear maps $\{\varphi_i^L: L^i \rightarrow L\}$ such that
\begin{enumerate}
\item $L^i \supset L^{i+1}, L^{a} = L, L^{a+\ell} = 0$ and each $L^i$ is a direct summand of $L$,

\item $\varphi_i^L|_{L^{i+1}} = \ell\varphi_{i+1}^L$ and $L = \sum_i \varphi_i^L(L^i)$.
\end{enumerate}
Then Fontaine-Laffaille \cite{Fontaine-Laffaille} gives an equivalence of categories between $\mathrm{MF}^{a,a+\ell}(\OO_K)$ and the category of $\OO_E[G_K]$-modules that are finitely generated subquotients of $E$-linear crystalline $G_K$-representations $V$ with Hodge-Tate weights in the interval $[-(a+\ell -1), -a]$.

\begin{rem}
Here we use the definition of the Tate twist of a strongly divisible lattice as in Section 4 of \cite{Bloch-Kato} to extend the results of \cite{Fontaine-Laffaille} to the case where $a \neq 0$.
\end{rem}

\begin{expl}\label{unramex}
Let $\psi: G_K \rightarrow \OO_E$ be an unramified character of finite order and let $\OO_E(\psi)$ denote the $\OO_E[G_K]$-module of rank one with $G_K$-action given by $\psi$.  Then the strongly divisible lattice $D_\psi$ corresponding to $\OO_E(\psi)$ can be described as a free rank one $\OO$-module such that $L_\psi^0 = L_\psi$, $L_\psi^1 = 0$ and $\varphi_0^{L_\psi}$ is multiplication by some $u \in \OO^\times = (\OO_E^\times)^{[K:\QQ_\ell]}$.  Denote this $u$ by $\psi(\sigma)$.  

We adopt this notation since over some finite extension of $E$, we have that $L_\psi$ is isomorphic to a strongly divisible lattice $L$ where $\varphi_0^L$ is multiplication by $(\psi(\sigma), 1,\dots, 1)$ (see \cite{Dousmanis}).  In any case, the precise value of $u$ will not be important for our intended application.
\end{expl}

For the remainder of the section, we assume $\ell$ is unramified in $F$ (a totally real extension of $\QQ$ of degree $d$) and set $K = F_v$ for a place $v$ of $F$ dividing $\ell$.

\begin{expl}\label{modformex}
Let $f$ be a newform on $F$ in $S_k(\mf{n},\psi)$.  For a prime $v|\ell$ of $F$, let $E = K_{f,\lambda}F_v$ and consider the Galois representation $\rho_{f,\lambda}|_{G_v} : G_v \rightarrow \GL_2(E)$.  Let $V_{f,\lambda}$ be a 2-dimensional $E$ vector space on which $G_v$ acts by $\rho_{f,\lambda}|_{G_v}$.  Fix a $G_v$-stable $\OO_E$-lattice $T_{f,\lambda} \subset V_{f,\lambda}$.  Note that $V_{f,\lambda}$ is crystalline with labeled Hodge-Tate weights $(-\frac{k_0 -2+k_{\tau_i}}{2},-\frac{k_0 - k_{\tau_i}}{2})_i$ if $\ell>k_0$ is unramified in $F$ and prime to $\mf{n}$.  Thus for $v$ dividing such $\ell$, there is a $L_f$ in $\mathrm{MF}^{0,\ell}(\OO_{F_v})$ corresponding to $T_f$.  Then using the Hodge-Tate weights, we have the following description of $L_f$.  

Set $T_0 = \{j | k_{\tau_j} = k_0\}$.  For $i\geq 1$, define $k_i$ to be the $\max_{j\not\in T_{i-1}}\{k_{\tau_j}\}$ and set $T_i = \{j | k_{\tau_j} = k_i\}.$  Let $s$ be the index such that $k_s = \min_i \{k_{\tau_i}\}$.  Set $e_i$ denote the element of $\OO = \OO_E^{[F_v:\QQ_\ell]}$ with a 1 in its $i$th component and zeroes everywhere else.  Finally, define $$d_{\geq i} = \sum_{j\in \cup_{k\geq i}T_i} e_j \mbox{ \ and \ } d_{ <i} = \sum_{j\not\in \cup_{k < 1} T_i} e_j.$$  Then there is an $\OO$ basis $x,y$ of $L_f$ such that the filtration satisfies:
$$L_f^i = \left\{\begin{array}{ll} \OO x \oplus \OO y, & \mbox{ for } i \leq 0,\\ \OO x \oplus \OO d_{\geq 1}y, & \mbox{ for } 1 \leq i \leq \frac{k_0-k_{1}}{2},\\ \OO  x \oplus \OO d_{\geq 2}y, & \mbox{ for } \frac{k_0 - k_{1}}{2} + 1 \leq i \leq \frac{k_0-k_{2}}{2},\\ \vdots & \vdots \\ \OO x, & \mbox{ for } \frac{k_0 - k_{s}}{2} + 1 \leq i \leq \frac{k_0 -2+k_{\tau_s}}{2},\\ \OO d_{<s}x, & \mbox{ for } \frac{k_0 -2+ k_{s}}{2} + 1 \leq i \leq \frac{k_0 -2+k_{s-1}}{2}\\ \vdots & \vdots\\\OO d_{<1}x, & \mbox{ for } \frac{k_0 -2+ k_{1}}{2} + 1 \leq i \leq k_0 - 1, \\  0, & \mbox{ for } i \geq k_0.\end{array}\right.$$  This is not enough to completely identify $L_f$ up to isomorphism, but it will be enough for our purposes.  

We fix some notation for use in Proposition \ref{l=p}.  Let $$\varphi_0^{L_f}(x) = \alpha x + \beta y$$ for some $\alpha, \beta$  in $\OO$.  So writing $\alpha = (\alpha_i)$ and $\beta = (\beta_i)$, we conclude that the $\lambda$-adic valuations $v_\lambda(\alpha_i)$ and $v_\lambda(\beta_i)$ are at least $k_0-1$ for all $i$ by condition 2 of the definition of the objects of $\mathrm{MF}^{0,\ell}(\OO_{F_v})$.  Here we normalized $v_\lambda$ so that $v_\lambda(\ell) = 1$.
\end{expl}

\begin{prop}\label{l=p}
Suppose $f$ is a Hilbert newform on $F$ of weight $k=(k_1,\dots,k_d)$, level $\mf{n}$, and character $\psi$.  Assume at least one $k_i > 2$ and set $k_0 = \max\{k_i\}$.  Then for $\ell > 2k_0$, unramified in $F$ and prime to $\mf{n}$,  $$H^0(G_v, \bar\varepsilon\otimes \ad^0\bar\rho_{f,\lambda}) = 0.$$
\end{prop}

\begin{proof}
We retain the notation from Examples \ref{unramex} and \ref{modformex}.  Since $\det \rho_{f,\lambda} = \psi_f^{-1}\varepsilon^{1-k_0}$, we have the Galois-stable lattice $$\ad^0T_f(1) \cong (\Sym^2(T_f) \otimes_{\OO_E}\OO_E(\psi_f))(k_0).$$  Since $\ell > 2k_0$, we can apply the Fontaine-Laffaille functor to get a corresponding $L$ in $\mathrm{MF}^{-k_0,k_0-1}(\OO_{F_v})$.  By \cite[Proposition 1.7]{Fontaine-Messing}, $$L \cong (\Sym^2(L_f) \otimes_\OO L_\psi)(k_0).$$  Note that by \cite[Lemma 4.5]{Bloch-Kato}, \begin{equation}\label{l=pvanisheq}H^1(G_v, \ad^0\bar\rho_{f,\lambda}) \cong \Ker(1 - \varphi_0^L:L^0/\lambda L^0 \rightarrow L/\lambda L).\end{equation}  Furthermore, by the definition of Tate twists for strongly divisible lattices,
\begin{eqnarray*}
L^0 & = &(\Sym^2(L_f) \otimes_\OO L_{\psi_f})^{k_0}\\
& = &\left\{v\otimes w\left| v = \sum_i a_i(u\otimes u') \in \Sym^2(L_f), u \in L_f^i, u' \in L_f^{i'}, i + i' = k_0\right.\right\}\\
& = &L_f^1\otimes_\OO L_f^{k_0-1} \otimes_\OO L_{\psi_f}^0\\
& = &\OO d_{<1}(x\otimes x \otimes w)
\end{eqnarray*}
where $w$ is a generator of the rank one $\OO$-module $L_{\psi_f}$.  Set $v = x\otimes x \otimes w$.  Then we have 
\begin{eqnarray}
\nonumber \varphi_0^L(ad_{<1}v)& = &a^\sigma d_{<1}^\sigma\varphi_1^{L_f}(x) \otimes \varphi_{k_0-1}^{L_f}(x) \otimes \varphi_0^{L_{\psi_f}}(w)\\
& = &\frac{(ad_{<1})^\sigma\psi_f(\sigma)\alpha^2}{\ell^{k_0}}x\otimes x\otimes w + \cdots\label{phi_0}
\end{eqnarray} where the superscript $\sigma$ denotes the action of $1_{\OO_E} \otimes \sigma$ on the given element of $\OO$.  Suppose that (\ref{l=pvanisheq}) is nonzero.  Thus if $ad_{<1}(x\otimes x\otimes w)$ is a nonzero element of the kernel of $1 - \varphi_0^L$ then (\ref{phi_0}) implies that there is some $i$ such that $$\frac{(ad_{<1})_i^\sigma\psi_f(\sigma)\alpha_i^2}{\ell^{k_0}} \equiv (ad_{<1})_i \not\equiv 0 \bmod \lambda.$$  This implies that the $\lambda$-adic valuation of the numerator is $k_0$.  As $v_\lambda(a_i\psi_f(\sigma))= 0$, this means that $$v_\lambda(\alpha_i^2) = k_0.$$  But we also know that $v_\lambda(\alpha_i)\geq k_0-1$, so we have that $2k_0-2 \leq k_0$.   As we assumed $k_0 > 2$, this proves the proposition.
\end{proof}

We are now ready to prove Theorem \ref{mainthm}.

\begin{thm*}
Let $f$ be as in Proposition $\ref{l=p}$ and suppose that it does not have CM and that it is not a twist of a base change of a Hilbert newform on $E \subsetneq F$.  Then $H^2(G_{F,S}, \ad^0\bar\rho_{f,\lambda}) = 0$ for almost all primes $\lambda$ of $K_f$.  That is, $\DD_{\bar\rho_{f,\lambda}}^{\det=\delta}$ is unobstructed for almost all $\lambda$.
\end{thm*}

\begin{proof}
We verify that the hypotheses (1) -- (3) of Proposition \ref{unobscriterion} hold for almost all $\lambda$.  Combining the results of Corollary \ref{princeseriescor}, Proposition \ref{specialcase} and Proposition \ref{l=p} shows that $H^0(G_v, \ol{\varepsilon}\otimes\ad^0\ol{\rho}_{f,\lambda}) = 0$ for almost all $\lambda$.  The Selmer group $H^1_f(G_F, A_{\rho_{f,\lambda}})$ vanishes for almost all $\lambda$ because of \cite[Theorem 6.6]{Dimitrov3}, \cite[Theorem 2.1]{Dimitrov4} and the fact that (in Dimitrov's notation) $$H_f^1(G_F, A_{\rho_{f,\lambda}}) \subset H^1_\Sigma(G_F, A_{\rho_{f,\lambda}})$$ for any finite set of primes $\Sigma$.  

Theorem B(i) of \cite{Dimitrov3} and \cite[Theorem 2.1]{Dimitrov4} tell us that $H_f^1(G_F, V_{\rho_{f,\lambda}}) = 0$ for almost all $\lambda$.  To show the vanishing of $H_f^1(G_F, V_{\rho_{f,\lambda}}(1))$, we define for a place $v |\ell$ of $F$ the tangent space $$t_V = ((B_\crys/B_\crys^+) \otimes_{\QQ_\ell} V)^{G_{F_v}}$$ of a crystalline $E$-linear representation $V$ where $E$ is a finite extension of $\QQ_\ell$ containing $K_{f,\lambda}$ and $F_v$.    Then \cite[Proposition I.2.2.2(ii)]{Fontaine-Perrin} tells us that $$t_V \cong D_\crys(V)/D_\crys(V)^0.$$  In particular, for $V_{\rho_{f,\lambda}}$, we extend scalars to a finite extension $E$ of $\QQ_\ell$ such that $E$ contains $F_v$ for all places $v | \ell$ in $F$ and set $$t_{V_{\rho_{f,\lambda}}} = \bigoplus_{v|\ell} t_{V_{\rho_{f,\lambda}|_{G_v}}}.$$  By Schur's lemma, $H^0(G_F, \varepsilon\otimes\ad^0\rho_{f,\lambda}) = 0$, so \cite[Remark II.2.2.2]{Fontaine-Perrin} implies that $$\dim_E H_f^1(G_F, V_{\rho_{f,\lambda}}(1)) - \dim_E H_f^1(G_F, V_{\rho_{f,\lambda}}) = -\dim_E t_{V_{\rho_{f,\lambda}}} + \sum_{v|\infty} H^0(G_v, V_{\rho_{f,\lambda}}).$$  A straightforward computation using the Hodge-filtration on $V_{\rho_{f,\lambda}}$ shows that the right-hand-side vanishes, so $\dim_EH_f^1(G_F,V_{\rho_{f,\lambda}}(1)) = \dim_E H_f^1(G_F, V_{\rho_{f,\lambda}})$.  Thus Proposition \ref{unobscriterion} implies the unobstructedness of $\DD_{\bar\rho_{f,\lambda}}^{\det=\delta}$ for almost all $\lambda$.
\end{proof}

\section{Explicit computations}

The methods we used to prove Theorem \ref{mainthm} are essentially effective in the sense that given enough information about the Hecke eigenvalues of a given Hilbert newform as well as the eigenvalues of the other newforms of the same level, one can find an explicit lower bound $B$ such that for all $\ell\geq B$, the deformation problem $\DD_{\bar\rho_{f,\lambda}}^{\det=\delta}$ is unobstructed for all $\lambda$ over these $\ell$.  We illustrate this with an example.

Let $F = \QQ(\sqrt{5})$, $k = (2,4)$, and $\mf{n} = (3+\omega)$ where $\omega = \frac{1+\sqrt{5}}{2}$.  Then using MAGMA we computed that the space of cuspforms $S_k(\mf{n})$ is one dimensional and that, moreover, there are no cuspforms of lower level.  Thus $S_k(\mf{n})$ is generated by a newform $f$ whose first few Hecke eigenvalues $c(f,\p)$ we computed in MAGMA and list in Table 1.

\begin{table}[h]
\begin{center}\caption{Hecke eigenvalues of $f$}\vspace{0cm}
\begin{tabular}{|c|c|c|c|c|c|}
\hline
$\pi$ & 2 & $\sqrt{5}$ & 3 & $4 - \omega$ & $4+\omega$ \\
\hline
$c(f,\p)$& $-2\sqrt{5} - 10$ & $5\sqrt{5} - 5$ & $-6\sqrt{5}$ & $-15\sqrt{5} + 17$ & $44\sqrt{5} - 60$ \\
\hline\hline
$\pi$ & $5-\omega$ & $5+\omega$& $6 - \omega$ & $5\omega - 2$ & $5\omega - 3$  \\
\hline
$c(f,\p)$ & $15\sqrt{5} + 55$ & $14\sqrt{5} -20$ & $-58\sqrt{5}$ & $-15\sqrt{5} - 3$  & $-30\sqrt{5} - 118$ \\ 
\hline
\end{tabular}
\end{center}
\end{table}

\begin{rem}
It can be shown that $K_f = \QQ(\sqrt{5})$ in this case.
\end{rem}

\begin{rem}
In what follows, note that $S = \{\mf{n}, \lambda\} \cup \{v|\infty\}$ for $\ol{\rho}_{f,\lambda}$.
\end{rem}

\begin{prop}
The deformation problem $\DD_{\bar\rho_{f,\lambda}}^{\det=\delta}$ is unobstructed for all primes $\lambda$ of $K_f$ over $\ell\geq 11$ and not dividing $\mf{n}$.
\end{prop}

\begin{proof}
Our approach is to give a lower bound on $\ell$ for which the residual representation $\ol{\rho}_{f,\lambda}$ is absolutely irreducible and for which the three hypotheses of Proposition \ref{unobscriterion} hold.  We begin with absolute irreducibility.  By \cite[Proposition 3.1(ii)]{Dimitrov1}, since $\omega^{10} - 1 \in \mf{n}$, we conclude that $\ol{\rho}_{f,\lambda}$ is absolutely irreducible for all $\lambda$ not dividing $\omega^{10} -1$, $\omega^{20}-1$, $\omega^{40} - 1$, and $\omega^{50} - 1$.  More concretely, computing the prime factors of the principal ideals generated by these elements, \cite[Proposition 3.1(ii)]{Dimitrov1} tells us that $\ol{\rho}_{f,\lambda}$ is absolutely irreducible for all $\lambda \nmid \mf{n}$ over $\ell \geq 11$ except possibly $(4 - \omega)$ and the primes over $41, 101$ and $151$.  

Note that since $\ol{\rho}_{f,\lambda}$ is an odd representation, it is absolutely irreducible if and only if it is irreducible.  Thus to prove absolute irreducibility for $\ol{\rho}_{f,\lambda}$, it suffices to provide a prime $\p$ over $p \nmid 11\ell$ such that the characteristic polynomial of $\ol{\rho}_{f,\lambda}(\Frob_\p)$ is irreducible over $k_{f,\lambda}$.  Recall that the characteristic polynomial for $\ol{\rho}_{f,\lambda}(\Frob_\p)$ is $X^2 - c(f,\p)X + p^3$ if $p$ splits in $F$ and is $X^2 - c(f,\p)X + p^4$ if $p$ is inert in $F$.  In particular, for each $\lambda$ over $41, 101$, and $151$, we found a prime $\p$ over 29 such that the polynomial $X^2 - c(f,\p)X + 29^3$ is irreducible over $k_{f,\lambda}$.  For $\lambda = (4 - \omega)$, we similarly computed the Hecke eigenvalue using MAGMA for $\p = (7)$ and found that $X^2 - (91\sqrt{5} - 35)X + 7^4$ is irreducible over $\F_{11} = \OO_F/(4-\omega)$.  Thus $\ol{\rho}_{f,\lambda}$ is absolutely irreducible for all $\lambda \nmid \mf{n}$ over $\ell \geq 11$.

We now check hypotheses (1) -- (3) of Proposition \ref{unobscriterion}.  For (1), we know that $H^0(G_\lambda, \ol{\varepsilon}\otimes\ad^0\ol{\rho}_{f,\lambda}) = 0$ for all $\lambda\nmid \mf{n}$ over $\ell > 8$ by Proposition \ref{l=p}.  Furthermore, for $\mf{n}$, we know that the local component $\pi_\mf{n}$ of the automophic representation corresponding to $f$ is special since $f$ is new at $\mf{n}$ and has trivial nebentypus.  In particular, using the second paragraph of the proof of Proposition \ref{specialcase}, we see that $H^0(G_{\mf{n}}, \ol{\varepsilon}\otimes\ad^0\ol{\rho}_{f,\lambda}) = 0$ for $\lambda \nmid 2, 3, 5,$ and 11 since there are no modular forms of weight (2,4) on $\QQ(\sqrt{5})$ of lower level.

Regarding (2), as we discussed in the proof of Theorem \ref{mainthm}, the vanishing of $H_f^1(G_F, V_{\ol{\rho}_{f,\lambda}}(1))$ is equivalent to the condition that $H_f^1(G_F, V_{\ol{\rho}_{f,\lambda}})= 0$.  Furthermore, Dimitrov \cite[Theorem B]{Dimitrov3} showed that $H_f^1(G_F, V_{\ol{\rho}_{f,\lambda}})$ vanishes as long as $\ell > 5$, $\lambda \nmid \mf{n}$, and the image of $\Ind_F^\QQ\ol{\rho}_{f,\lambda}$ is ``large.''  (We will give more details about this large image condition in the next paragraph.)  This means that for the desired set of primes $\lambda$, whenever this ``large image'' condition holds, hypothesis (2) of Proposition \ref{unobscriterion} also holds.  Moreover, let $\eta_f$ denote the congruence ideal obtained from the $\oo$-algebra homomorphism $\mathcal{T}_\mf{m} \rightarrow \oo$ by $T_\mf{a} \mapsto \iota_\lambda(c(f,\mf{a}))$ where $\oo = \OO_{K_f,\lambda}$, $\mathcal{T}$ is the Hecke algebra $\oo[T_\mf{a}| \mf{a} \subset \OO_F]$ and $\iota_\lambda$ is the fixed isomorphism $\CC \rightarrow \ol{K}_{f,\lambda}$ extending the embedding $\OO_{K_f} \hookrightarrow \oo$.  (See \cite[Definition 3.1]{Dimitrov3} as well as the discussions before \cite[Theorems 1.4 and 3.6]{Dimitrov3} for more details about $\eta_f$.)  Then \cite[Theorem 3.6]{Dimitrov3} implies that for $\ell > 5, \lambda \nmid \mf{n}$, and $\Ind_F^\QQ \ol{\rho}_{f,\lambda}$ satisfying the same large image hypothesis, the Selmer group $H_f^1(G_F, A_{\ol{\rho}_{f,\lambda}}) = 0$ if and only if $\eta_f = \oo$.  That is, if and only if $\lambda$ does not divide $\eta_f$.  By definition, however, $\lambda$ divides $\eta_f$ if and only if there is another newform $g$ of the same weight, level and character such that $c(f,\mf{a}) \equiv c(g,\mf{a}) \bmod \lambda$ for all $\mf{a} \subset \OO_F$.  As $f$ is the only newform in $S_k(\mf{n})$, this means that $H_f^1(G_F, A_{\ol{\rho}_{f,\lambda}}) = 0$ for all such $\lambda$.  That is, hypothesis (3) also holds for all $\lambda$ over $\ell > 5, \lambda\nmid \mf{n}$ satisfying the large image condition.  Thus we are reduced to checking that this large image condition on $\Ind_F^\QQ\ol{\rho}_{f,\lambda}$ holds for all $\lambda$ over $\ell \geq 11$ such that $\lambda \neq \mf{n}$.

The large image condition on $\Ind_F^\QQ\ol{\rho}_{f,\lambda}$ that we referred to throughout the previous paragraph is a somewhat technical hypothesis that Dimitrov uses for Theorem 1.4 of \cite{Dimitrov3}.  We refer the interested reader to \cite[Theorem A]{Dimitrov3} for a detailed statement of this large image hypothesis on $\Ind_F^\QQ\ol{\rho}_{f,\lambda}$.  In our case, however, since the weight (2,4) is non-induced in the sense of \cite[Definition 3.11]{Dimitrov1} and we assume that $\ell \geq 11$, we may instead use the large image condition on $\ol{\rho}_{f,\lambda}$ that $\image(\ol{\rho}_{f,\lambda})$ contains a conjugate of $\SL_2(k_{f,\lambda})$ (see \cite[Proposition 3.13]{Dimitrov1}).  Moreover, since we have already shown that $\ol{\rho}_{f,\lambda}$ is irreducible for all $\lambda\nmid \mf{n}$ over $\ell \geq 11$, we can use Dickson's classification of subgroups of $\GL_2(k_{f,\lambda})$ in such cases.  In particular, this classification states that an irreducible subgroup of $\GL_2(k_{f,\lambda})$ that does \emph{not} contain a conjugate of $\SL_2(k_{f,\lambda})$ is isomorphic to either a dihedral group or one of $A_4, S_4,$ or $A_5$.  Thus we need to show that the projective image of $\image(\ol{\rho}_{f,\lambda})\subset \GL_2(k_{f,\lambda})$ is not isomorphic to a dihedral group nor any of the groups $A_4, S_4,$ and $A_5$.

To check that the projective image of $\image(\ol{\rho}_{f,\lambda})$ is not dihedral, we use \cite[Lemma 3.4]{Dimitrov1}.  More specifically, assume that the image of $\ol{\rho}_{f,\lambda}$ in $\PGL_2(k_{f,\lambda})$ is dihedral, meaning $\ol{\rho}_{f,\lambda} \cong \ol{\rho}_{f,\lambda} \otimes \chi_{K/F}$ where $\chi_{K/F}$ is the character of a quadratic extension $K/F$.  Then supposing that $\ell \neq 2k_i -1$ for all $i$ where $k = (k_1,\dots, k_d)$ is the weight of $f$, this lemma says  that $K/F$ is unramified outside of $\mf{n}$.  Thus we have the following method for showing that the image of $\ol{\rho}_{f,\lambda}$ in $\PGL_2(k_{f,\lambda})$ is not dihedral.  For each quadratic field $K$ unramified outside of $\mf{n}$, we find primes $\p$ and $\mf{q}$ of $F$ that are inert in  $K$ and such that $c(f,\p)$ and $c(f,\mf{q})$ do not lie in a common $\lambda\neq \mf{n}$ over $\ell \geq 11$.  More concretely, there is a unique quadratic extension of $F$, unramified outside of $\mf{n}$, namely the ray class field $K = F(\sqrt{\omega(3+\omega)})$ for the modulus $\mf{n}\mf{m_\infty}$ where $\mf{m}_\infty$ contains all of the archimedean places of $F$.  We found that the ideals $\p = (5+\omega)$ and $\mf{q} = (5\omega -2)$ are inert in $K$.  Furthermore, the Hecke eigenvalues for these primes are $c(f,\p) = 14\sqrt{5} - 20$ and $c(f,\mf{q}) = -15\sqrt{5} - 3$.  As the prime divisors of $\mathrm{N}_{F/\QQ}c(f,\p)$ are 2, 5, and 29 while the prime divisors of $\mathrm{N}_{F/\QQ}c(f,\mf{q})$ are 2,3, and 31, we see that for each $\lambda \neq \mf{n}$ over $\ell \geq 11$, there is some prime $\mf{P}$ of $F$ that is inert in $K$ and $c(f,\mf{P})\not\equiv 0 \bmod \lambda$.  Hence the image of $\ol{\rho}_{f,\lambda}$ in $\PGL_2(k_{f,\lambda})$ is not dihedral.

Finally, to show that the projective image of $\ol{\rho}_{f,\lambda}$ is not isomorphic to $A_4, S_4,$ or $A_5$, we use Section 3.2 of \cite{Dimitrov1}.  The main result of this section is that if $$\ell -1 > \frac{5}{d} \sum_{i = 1}^d k_i -1$$ where $d = [F:\QQ]$ and $k=(k_1,\dots, k_d)$ is the weight of the newform $f$, then the projective image of $\ol{\rho}_{f,\lambda}$ is not isomorphic to any of the groups $A_4, S_4$ or $A_5$.  In our case, $d=2$ and $k= (2,4)$ so it is easy to conclude that the image of $\ol{\rho}_{f,\lambda}$ is not isomorphic to $A_4, S_4$, or $A_5$ for $\lambda$ over $\ell\geq 13$.  For $\lambda = (4-\omega)$, a closer analysis of \cite[Section 3.2]{Dimitrov1} shows that if the projective image of $\ol{\rho}_{f,\lambda}$ is $A_4, S_4,$ or $A_5$ then the arguments there imply that either $\pm1$ or $\pm 3$ has order $\leq5$ in $\ZZ/10\ZZ$, which is a contradiction.
\end{proof}

\bibliographystyle{plain}
\bibliography{ltbd}

\end{document}